\documentclass[a4paper, 10pt]{amsart}
\usepackage{amsmath}
\usepackage{amssymb}

\def\id{\operatorname{id}}
\def\re{\operatorname{Re}}

\newtheorem{cor}{Corollary}

\newtheorem{lem}[cor]{Lemma}
\newtheorem{thm}[cor]{Theorem}

\theoremstyle{remark}
\newtheorem{rem}[cor]{Remark}

\title{Proper holomorphic mappings between symmetrized ellipsoids}
\author{Pawe\l{} Zapa\l owski}

\address{Institute of Mathematics, Jagiellonian University, \L ojasiewicza 6 30-348 Krak\'ow, Poland}
\email{Pawel.Zapalowski@im.uj.edu.pl}

\begin{document}

\thanks{
\newline
The research was partially supported by the Research Grant No. N
N201 361436 of the Polish Ministry of Science and Higher
Education.}

\begin{abstract}We characterize the existence of proper holomorphic
mappings in the special class of bounded $(1,2,\dots,n)$-balanced
domains in $\mathbb C^n$, called the symmetrized ellipsoids. Using
this result we conclude that there are no non-trivial proper
holomorphic self-mappings in the class of symmetrized ellipsoids.
We also describe the automorphism group of these domains.
\end{abstract}
\maketitle

\section{Introduction and statement of results}

For $n\geqslant 2$ and $p>0$ let
\begin{equation*}
\mathbb B_{p,n}:=\{(z_1,\dots,z_n)\in\mathbb
C^n:\sum_{j=1}^n|z_j|^{2p}<1\}
\end{equation*}
denote the \emph{generalized complex ellipsoid}. We shall write
$\mathbb B_n:=\mathbb B_{1,n}$, $\mathbb T:=\partial\mathbb B_1$.
Note that $\mathbb B_{p,n}$ is bounded, complete Reinhardt domain.

Let $\pi_n=(\pi_{n,1},\dots,\pi_{n,n}):\mathbb
C^n\rightarrow\mathbb C^n$ be defined as follows
\begin{equation*}
\pi_{n,k}(z)=\sum_{1\leqslant j_1<\dots<j_k\leqslant
n}z_{j_1}\dots z_{j_k},\quad 1\leqslant k\leqslant n,\
z=(z_1,\dots,z_n)\in\mathbb C^n.
\end{equation*}
Note that $\pi_n$ is a proper holomorphic mapping with
multiplicity $n!$, $\pi_n|_{\mathbb B_{p,n}}:\mathbb
B_{p,n}\rightarrow\pi_n(\mathbb B_{p,n})$ is proper too.

The set
\begin{equation*}
\mathbb E_{p,n}:=\pi_n(\mathbb B_{p,n})
\end{equation*}
is called the \textit{symmetrized} $(p,n)$-\textit{ellipsoid}.
Note that $\mathbb E_{p,n}$ is bounded $(1,2,\dots,n)$-balanced
domain (recall that a domain $D\subset\mathbb C^n$ is called the
$(k_1,\dots,k_n)$-\emph{balanced}, where $k_1,\dots,k_n\in\mathbb
N$, if $(\lambda^{k_1} z_1,\dots,\lambda^{k_n}z_n)\in D$ for any
$(z_1,\dots,z_n)\in D$ and $\lambda\in\overline{\mathbb B}_1$).
Geometric properties of $\mathbb E_{p,n}$ were studied in
\cite{z2008}. Here we answer some of the open questions posed
there. As the definition of the symmetrized ellipsoid is similar
to the one of the symmetrized polydisc $\mathbb G_n:=\pi_n(\mathbb
B_1^n)$, which has drawn a lot of attention recently (see
\cite{ay2004}, \cite{c2004}, \cite{ez2005}, \cite{npz2006} and the
references given there), it is quite natural to ask which
properties of the symmetrized polydisc are inherited by the
symmetrized ellipsoids.

Our aim is to give necessary and sufficient condition for
existence of the proper holomorphic mappings between the
symmetrized ellipsoids, the class of bounded
$(1,\dots,n)$-balanced domains.

Here is some notation. Let $\mathfrak S_n$ denote the group of
permutations of the set $\{1,\dots,n\}$. For $\sigma\in\mathfrak
S_n,\ z=(z_1,\dots,z_n)\in\mathbb C^n$ denote
$z_{\sigma}:=(z_{\sigma(1)},\dots,z_{\sigma(n)})$. Next, for any
$A\subset\mathbb C$ put $A_*:=A\setminus\{0\}$, $A^n_*:=(A_*)^n$.
Moreover, for any $z=(z_1,\dots,z_n)\in\mathbb C^n$,
$w=(w_1,\dots,w_n)\in\mathbb C^n$, $t\in\mathbb C$ and $r>0$ we
put $zw:=(z_1w_1,\dots,z_nw_n)$, $tz:=(tz_1,\dots,tz_n)$, and
$z^r:=(z_1^r,\dots,z_n^r)$.

\begin{rem}\label{rem:sym}
(a) Let $l\in\mathbb N$. Observe that $\mathbb C^n\ni
z\mapsto\pi_n(z^l)\in\mathbb C^n$ is a symmetric polynomial
mapping. According to the fundamental theorem of symmetric
polynomials (see e.g.~\cite{m1995}) there is a unique polynomial
mapping $P_l:\mathbb C^n\rightarrow\mathbb C^n$ such that
$\pi_n(z^l)=P_l(\pi_n(z))$, $z\in\mathbb C^n$. In particular,
$P_l(\mathbb E_{p,n})=\mathbb E_{p/l,n}$ for any $p>0$.

(b) Fix $A,B,C\in\mathbb C$ and put $L:=(L_1,\dots,L_n):\mathbb
C^n\rightarrow\mathbb C^n$, where
\begin{equation*}
L_j(z):=A\sum_{k=1}^nz_k+Bz_j+C,\quad z=(z_1,\dots,z_n)\in\mathbb
C^n,\ j=1,\dots,n.
\end{equation*}
Observe that $\pi_n\circ L$ is a symmetric polynomial mapping.
According to the fundamental theorem of symmetric polynomials
there is a unique polynomial mapping $S_L:\mathbb
C^n\rightarrow\mathbb C^n$ such that $\pi_n\circ L=S_L\circ\pi_n$.
\end{rem}
Now we are in position to formulate our main result.

\begin{thm}\label{thm:main}There exists proper holomorphic mapping
$f:\mathbb E_{p,n}\rightarrow\mathbb E_{q,n}$ iff $p/q\in\mathbb
N$. Furthermore, if $p/q\in\mathbb N$, the only proper
holomorphic mappings $f:\mathbb E_{p,n}\rightarrow\mathbb E_{q,n}$
\begin{enumerate}
\item[(a)]in case $p\neq1$, or $q\neq1/(2m)$, $m\in\mathbb N$, or $n\neq2$ are of the form
\begin{equation}\label{eq:maina}
f=P_{p/q}\circ\phi,
\end{equation}
where $P_{p/q}$ is as in Remark~\ref{rem:sym}~(a) and $\phi$ is an automorphism of $\mathbb E_{p,n}$;
\item[(b)] in case $p=1$, $q=1/(2m)$, $m\in\mathbb N$, and $n=2$ are of the form (\ref{eq:maina}) or
\begin{equation*}
f=P_m\circ\phi_{III}\circ P_2\circ\phi_{II},
\end{equation*}
where $\phi_{II}$ (resp. $\phi_{III}$) is the automorphism of
$\mathbb E_{1,2}$ (resp. $\mathbb E_{1/2,2}$) defined in
Corollary~\ref{cor:aute}.
\end{enumerate}
\end{thm}

Similar classification for the class of generalized complex
ellipsoids (with not necessarily equal exponents on each
coordinate) was done in \cite{l1984} (the case of positive integer
exponents) and \cite{ds1991} (case of positive real exponents).

An immediate consequence of Theorem~\ref{thm:main} is the
following Alexander-type theorem for the symmetrized ellipsoids
saying that every proper holomorphic self-map of the symmetrized
ellipsoid is an automorphism.

\begin{cor}\label{cor:alex}Let $f:\mathbb E_{p,n}\rightarrow\mathbb E_{p,n}$ be a
proper holomorphic self-mapping. Then $f$ is an automorphism.
\end{cor}

Theorem of that type was obtained in the case of $\mathbb B_n$ in
\cite{a1977} and its generalization on complex ellipsoids was done
in \cite{l1984} and \cite{ds1991}. Recently similar result was
obtained in \cite{k2009} for the tetrablock, which is
$(1,1,2)$-balanced domain in $\mathbb C^3$. Characterization of
proper holomorphic self-mappings of symmetrized polydisc is done
in \cite{ez2005}.

Furthermore, from the proof of Theorem~\ref{thm:main}, the
automorphisms group of $\mathbb E_{p,n}$ may be easily derived.

\begin{cor}\label{cor:aute}
\begin{enumerate}
\item[(a)] If $p\neq1$ and $(p,n)\neq(1/2,2)$ then the only automorphisms of $\mathbb E_{p,n}$ are
of the form
\begin{equation}\label{eq:e1}
\phi_I(z_1,z_2,\dots,z_n)=(\zeta
z_1,\zeta^2z_2,\dots,\zeta^nz_n),\quad(z_1,z_2\dots,z_n)\in\mathbb
E_{p,n},
\end{equation}
where $\zeta\in\mathbb T$.
\item[(b)]The only automorphisms of $\mathbb E_{1,n}$, are
of the form
\begin{equation}\label{eq:e2}
\phi_{II}(z)=\left(\frac{S_{L_{\varphi_{II}},1}(z)}{n(1-a_0z_1)},\dots,\frac{S_{L_{\varphi_{II}},n}(z)}{n^n(1-a_0z_1)^n}\right),\quad
z=(z_1,\dots,z_n)\in\mathbb E_{1,n},
\end{equation}
where
$S_{L_{\varphi_{II}}}=(S_{L_{\varphi_{II}},1},\dots,S_{L_{\varphi_{II}},1})$
is the polynomial mapping as in Remark~\ref{rem:sym}~(b) induced
by
$L_{\varphi_{II}}=(L_{\varphi_{II},1},\dots,L_{\varphi_{II},n}):\mathbb
C\rightarrow\mathbb C^n$, where
\begin{equation*}
L_{\varphi_{II},j}(z_1,\dots,z_n):=\zeta_1\Big(\sum_{k=1}^nz_k-na_0\Big)+\zeta_2\sqrt{1-na_0^2}\Big(\sum_{k=1}^nz_k-nz_j\Big),
\end{equation*}
for some $\zeta_1,\zeta_2\in\mathbb T$, $a_0\in\mathbb R$,
$a_0^2<\frac{1}{n}$.
\item[(c)]The only automorphisms of $\mathbb E_{1/2,2}$ are
of the form (\ref{eq:e1}) or
\begin{equation}\label{eq:e4}
\phi_{III}(z_1,z_2)=\left(\zeta
z_1,\zeta^2\left(\tfrac14z_1^2-z_2\right)\right),\quad(z_1,z_2)\in\mathbb
E_{1/2,2},
\end{equation}
where $\zeta\in\mathbb T$.
\end{enumerate}
\end{cor}

\begin{rem}It should be mentioned that the automorphisms of the form
(\ref{eq:e1}) are special cases of the automorphisms of the form
(\ref{eq:e2}).
\end{rem}

\textbf{Acknowledgements.} The author is greatly indebted to
\L.~Kosi\'nski for many stimulating conversations.

\section{Proofs}

\begin{rem}For a $(k_1,\dots,k_n)$-balanced domain $D\subset\mathbb C^n$ one may
define the \emph{generalized Minkowski functional}
\begin{equation*}
\mu_D(z_1,\dots,z_n):=\inf\{\lambda>0:(\lambda^{-k_1}z_1,\dots,\lambda^{-k_n}z_n)\in
D\},\quad (z_1,\dots,z_n)\in\mathbb C^n.
\end{equation*}
Observe that for $(1,2,\dots,n)$-balanced domain $\mathbb E_{p,n}$
we have
\begin{equation*}
\mu_{\mathbb
E_{p,n}}(z)=\max\Big\{\Big(\sum_{j=1}^n|w_j|^{2p}\Big)^{1/2p}:(w_1,\dots,w_n)\in\pi_n^{-1}(z)\Big\},\quad
z\in\mathbb C^n.
\end{equation*}
In particular, $\mu_{\mathbb E_{p,n}}$ is continuous.
\end{rem}

\begin{proof}[Proof of Theorem~\ref{thm:main}]If $p/q\in\mathbb N$
then $f(z):=P_{p/q}(z)$ is a proper holomorphic mapping between
$\mathbb E_{p,n}$ and $\mathbb E_{q,n}$.

Assume $f:\mathbb E_{p,n}\rightarrow\mathbb E_{q,n}$ is proper and
holomorphic. Since $\mu_{\mathbb E_{p,n}}$ is continuous, it
follows from \cite{k2009} that $f$ extends holomorphically past
the boundary, $\partial\mathbb E_{p,n}$, of $\mathbb E_{p,n}$.
Hence there is a domain $V\subset\mathbb C^n$ such that
\begin{itemize}
\item$V\cap\partial\mathbb E_{p,n}\neq\varnothing$,
\item the mapping $f|_V:V\rightarrow f(V)$ is biholomorphic,
\item the mappings $\pi_n|_{\pi_n^{-1}(V)}:\pi_n^{-1}(V)\rightarrow
V$ and $\pi_n|_{\pi_n^{-1}(f(V))}:\pi_n^{-1}(f(V))\rightarrow
f(V)$ are biholomorphic.
\end{itemize}
Since $\pi_n(\partial\mathbb B_{p,n})=\partial\mathbb E_{p,n}$, it
is not restrictive to assume that for a domain
$U:=\pi_n^{-1}(V)\subset\mathbb C^n$
\begin{itemize}
\item$U\cap\partial\mathbb B_{p,n}\neq\varnothing$,
\item the mapping $U\ni z\overset{g_p}\longmapsto z^p\in g_p(U)$
is well defined and biholomorphic,
\item the mapping $\pi_n^{-1}(f(\pi_n(U)))\ni
z\overset{g_q}\longmapsto z^q\in g_q(\pi_n^{-1}(f(\pi_n(U))))$ is
is well defined and biholomorphic.
\end{itemize}
Hence the mapping $\psi:=\pi_n^{-1}\circ
f\circ\pi_n|_U:U\rightarrow\psi(U)$ is well defined and
biholomorphic. Consequently, the mapping
$\varphi:=g_q\circ\psi\circ g_p^{-1}|_{g_p(U)}$ is holomorphic and
$\varphi|_{g_p(U)\cap\mathbb B_n}$ is biholomorphic. As
$\varphi(g_p(U)\cap\partial\mathbb B_n)\subset\partial\mathbb
B_n$, it follows from \cite{a1974} that $\varphi$ extends to an
automorphism of $\mathbb B_n$, still denoted by
$\varphi=(\varphi_1,\dots,\varphi_n)$. Hence
\begin{equation}\label{eq:U}
\pi_n(\varphi^{1/q}(z))=f(\pi_n(z^{1/p})),\quad z\in g_p(U).
\end{equation}

We use the following lemma which will be proved afterwards.

\begin{lem}\label{lem:rotperm}Let $\varphi$ be an automorphism of $\mathbb
B_n$ which satisfies (\ref{eq:U}) and let $m:=1/q$, $l:=1/p$.
\begin{enumerate}
\item[(a)]If $m\notin\mathbb N$ then $m/l\in\mathbb N$ and, up to permutation of variables, $\varphi$ is of the form
\begin{equation}\label{eq:lI}
\varphi_{I}(z_1,\dots,z_n)=\zeta(\eta_1z_1,\dots,\eta_nz_n),\quad
(z_1,\dots,z_n)\in\mathbb B_n,
\end{equation}
for some $\zeta,\eta_j\in\mathbb T$, $\eta_j^m=1$, $j=1,\dots,n$.
\item[(b)]If $m\in\mathbb N$ then $l\in\mathbb N$ and $m/l\in\mathbb N$. Moreover,
\begin{enumerate}
\item[(i)]if $l=1$ and $n\geqslant3$ then, up to permutation of variables
and components, $\varphi$ is of the form
$\varphi_{II}=(\varphi_{II,1},\dots,\varphi_{II,n})$, where
\begin{multline}\label{eq:lII}
\varphi_{II,j}(z_1,\dots,z_n)\\=\frac{\eta_j}{n(1-a_0\sum_{k=1}^nz_k)}\left(\zeta_1\Big(\sum_{k=1}^nz_k-na_0\Big)+\zeta_2\sqrt{1-na_0^2}\Big(\sum_{k=1}^nz_k-nz_j\Big)\right),\\
(z_1,\dots,z_n)\in\mathbb B_n,
\end{multline}
for some $a_0\in\mathbb R$, $a_0^2<\frac{1}{n}$,
$\zeta_1,\zeta_2,\eta_j\in\mathbb T$, $\eta_j^m=1$, $j=1,\dots,n$;
\item[(ii)]if $l\geqslant2$ and $n\geqslant3$ then, up to permutation of variables
and components, $\varphi$ is of the form (\ref{eq:lI});
\item[(iii)]if $l=1$ and $n=2$ then, up to permutation of variables and components, $\varphi$ is of the
form (\ref{eq:lII}); moreover, if $m$ is even then, additionally,
up to permutation of variables and components, $\varphi$ is of the
form
\begin{multline}\label{eq:lIV}
\varphi_{III}(z_1,z_2)\\=\frac{1}{\sqrt{2}(1-a_0(z_1+z_2))}\left(\zeta_1(z_1+z_2-2a_0),\zeta_2\sqrt{1-2a_0^2}(z_1-z_2)\right),\\
\quad(z_1,z_2)\in\mathbb B_2,
\end{multline}
for some $\zeta_1,\zeta_2\in\mathbb T$ and $a_0\in\mathbb R$,
$a_0^2<\frac12$;
\item[(iv)]if $l=2$ and $n=2$ then, up to permutation of variables
and components, $\varphi$ is of the form (\ref{eq:lI}) or
\begin{equation}\label{eq:lV}
\varphi_{IV}(z_1,z_2)=\frac{\zeta}{\sqrt{2}}(z_1+z_2,\eta(z_1-z_2)),\quad
(z_1,z_2)\in\mathbb B_2,
\end{equation}
for some $\zeta,\eta\in\mathbb T$, $\eta^m=1$.
\end{enumerate}
\end{enumerate}
\end{lem}

\begin{rem}It should be mentioned that the automorphisms of the form
(\ref{eq:lI}) are special cases of the automorphisms of the form
(\ref{eq:lII}).
\end{rem}

Note that for any automorphism $\varphi$ of $\mathbb B_n$ which
satisfies (\ref{eq:lI}), (\ref{eq:lII}), or (\ref{eq:lV})
respectively, there is an automorphism $\tilde\varphi$ of $\mathbb
B_n$ such that $\pi_n(\varphi^m(z))=\pi_n(\tilde\varphi^m(z))$ and
$\pi_n(\tilde\varphi^l(z))=\pi_n(\tilde\varphi^l(z_{\sigma}))$ for
any $z\in\mathbb B_n$, and $\sigma\in\mathfrak S_n$.

Indeed,
\begin{itemize}
\item in case of (\ref{eq:lI}) it suffices to take $\tilde\varphi$
of the form (\ref{eq:lI}) with $\eta_j=1$, $j=1,\dots,n$. Then the
relation $\phi\circ\pi_n=\pi_n\circ\tilde\varphi$ defines the
automorphism $\phi_I$ of $\mathbb E_{p,n}$ of the form
(\ref{eq:e1}), which obviously satisfies the relation
\begin{equation}\label{eq:aut1}
\phi(\pi_n(z^{1/p}))=\pi_n(\tilde\varphi^{1/p}(z)),\quad
z\in\mathbb B_n.
\end{equation}
\item In case of (\ref{eq:lII}) it suffices to take $\tilde\varphi$
of the form (\ref{eq:lII}) with $\eta_j=1$, $j=1,\dots,n$. Then
the relation $\phi\circ\pi_n=\pi_n\circ\tilde\varphi$ defines the
automorphism $\phi_{II}$ of $\mathbb E_{1,n}$ of the form
(\ref{eq:e2}), which obviously satisfies the relation
(\ref{eq:aut1}).
\item In case of (\ref{eq:lV}) it suffices to take $\tilde\varphi$
of the form (\ref{eq:lV}) with $\eta^2=1$. Then the relation
(\ref{eq:aut1}), which in this case has form
\begin{equation*}
\phi(\pi_2(z^2))=\pi_2(\tilde\varphi^2(z)),\quad
z\in\mathbb B_2,
\end{equation*}
defines the automorphism $\phi_{III}$ of $\mathbb E_{1/2,2}$ of
the form (\ref{eq:e4}).
\end{itemize}

It follows from Lemma~\ref{lem:rotperm} that $p/q=m/l\in\mathbb
N$. Consequently, using (\ref{eq:aut1}),
\begin{align*}
f(\pi_n(z^{1/p}))&=\pi_n(\varphi^{1/q}(z))=\pi_n(\tilde\varphi^{1/q}(z))=\pi_n((\tilde\varphi^{1/p}(z))^{p/q})\\
{}&=P_{p/q}(\pi_n(\tilde\varphi^{1/p}(z)))=P_{p/q}(\phi(\pi_n(z^{1/p}))),\quad
z\in g_p(U).
\end{align*}
The identity principle implies that $f=P_{p/q}\circ\phi$ which
ends the proof in the case, when equality (\ref{eq:U}) is
satisfied by the automorphisms of the form (\ref{eq:lI}),
(\ref{eq:lII}), or (\ref{eq:lV}).

In the case when equality (\ref{eq:U}) is satisfied by the
automorphism of the form (\ref{eq:lIV}), the situation is slightly
different and we proceed as follows. First observe that
$\varphi_{III}=\varphi_{IV}\circ\varphi_{II}$, where
$\varphi_{IV}$ and $\varphi_{II}$ are taken with
$\eta=\eta_1=\eta_2=1$. Since $m$ is even, $m=2m'$ for some
$m'\in\mathbb N$. Then the previous cases imply
\begin{align*}
f(\pi_2(z))&=\pi_2(\varphi_{III}^{2m'}(z))=\pi_2(\varphi_{IV}^{2m'}(\varphi_{II}(z)))=P_{m'}(\pi_2(\varphi_{IV}^2(\varphi_{II}(z))))\\
{}&=P_{m'}(\phi_{III}(\pi_2(\varphi_{II}^2(z))))=P_{m'}(\phi_{III}(P_2(\phi_{II}(\pi_2(z))))),\quad
z\in g_1(U),
\end{align*}
whence $f=P_{m'}\circ\phi_{III}\circ P_2\circ\phi_{II}$.
\end{proof}

\begin{rem}\label{rem:aut}Following \cite{s1972} any automorphism
$\varphi=(\varphi_1,\dots,\varphi_n)$ of the unit ball is of the
form
\begin{equation*}\label{eq:aut}
\varphi_j(z)=\frac{\sum_{k=1}^nq_{j,k}(z_k-a_k)}{R(1-\sum_{k=1}^n\bar
a_kz_k)},\quad z=(z_1,\dots,z_n)\in\mathbb B_n,\ j=1,\dots,n,
\end{equation*}
where $a=(a_1,\dots,a_n)\in\mathbb B_n$ is arbitrary,
$Q=[q_{j,k}]$ and $R$ are respectively a $n\times n$ matrix and a
constant such that
\begin{equation*}
\bar Q(\mathbb I_n-\bar a{}^t\!a){}^t\!Q=\mathbb I_n,\qquad\bar
R(1-{}^t\!a\bar a)R=1,
\end{equation*}
where $\mathbb I_n$ is the unit $n\times n$ matrix, whereas $\bar
A$ (resp.~${}^t\!A$) is the conjugate (resp.~transpose) of an
arbitrary matrix $A$. Moreover, $a$, $Q$, and $R$ satisfy
\begin{equation}\label{eq:stein}
\begin{cases}{}^t\!Q\bar{Q}-|R|^2\bar a{}^t\!a=\mathbb
I_n\\|R|^2-{}^t\!a{}^t\!Q\bar{Q}\bar{a}=1\\{}^t\!Q\bar{Q}\bar{a}=|R|^2a
\end{cases}
.
\end{equation}
In particular, $Q$ is unitary if $a=0$.
\end{rem}

\begin{proof}[Proof of Lemma~\ref{lem:rotperm}]In the proof we will use the
form of automorphism $\varphi$ of $\mathbb B_n$ as in
Remark~\ref{rem:aut}.

\emph{Ad} (a). Assume $m\notin\mathbb N$. Note that the function
on the right side of (\ref{eq:U}) is well defined on any domain
$D\subset\mathbb B_n\cap\mathbb C^n_*$ such that the fiber
$D_j:=\{\lambda\in\mathbb
C:(z_1,\dots,z_{j-1},\lambda,z_{j+1},\dots,z_n)\in D\}$ is
connected and simply connected for $j=1,\dots,n$. In particular,
the function $D\ni x\mapsto(\prod_{j=1}^n\varphi_j(z))^m$ is
holomorphic. Assumption $m\notin\mathbb N$ implies that
\begin{equation}\label{eq:cstar}
\varphi(\mathbb B_n\cap\mathbb C^n_*)\subset\mathbb C^n_*.
\end{equation}
We show that $\varphi$ is of the form (\ref{eq:lI}).

First we show that $a=0$ and for any $j\in\{1,\dots,n\}$ there
exists a unique $k$ such that $q_{j,k}\neq0$.

Indeed, suppose the contrary. Then either
\begin{itemize}
\item there are $j,k_1,k_2$ with $k_1\neq k_2$ and $q_{j,k_1}\neq0\neq
q_{j,k_2}$, or
\item there are $j,k_1$ with $q_{j,k_1}\neq0\neq a_{k_1}$ (since $\varphi$ is one-to-one mapping, for any $k$ there is a $j$ such that $q_{j,k}\neq0$).
\end{itemize}
In both cases one may define
\begin{equation*}
w_{k_1}:=a_{k_1}-\sum_{k\neq
k_1}\frac{q_{j,k}}{q_{j,k_1}}(w_k-a_k)\neq0,
\end{equation*}
provided $w_k\in\mathbb C_*$, $k\neq k_1$, are chosen close to
$a_k$ enough. Clearly, one may assume that
$w:=(w_1,\dots,w_n)\in\mathbb B_n$. Consequently, $w\in\mathbb
B_n\cap\mathbb C^n_*$ with $\varphi_j(w)=0$---a contradiction with
(\ref{eq:cstar}).

First equality in (\ref{eq:stein}) implies that
$\varphi(z_1,\dots,z_n)=(\zeta_1z_{\sigma(1)},\dots,\zeta_nz_{\sigma(n)})$
for some $\sigma\in\mathfrak S_n$ and $\zeta_j\in\mathbb T$,
$j=1,\dots,n$. Moreover, (\ref{eq:U}) and the identity principle
imply that for any $\omega\in\mathfrak S_n$ there is
$\tau\in\mathfrak S_n$ such that
\begin{equation*}
(\zeta^m_1z^m_{\omega(1)},\dots,\zeta^m_nz^m_{\omega(n)})=
(\zeta^m_{\tau(1)}z^m_{\tau(1)},\dots,\zeta^m_{\tau(n)}z^m_{\tau(n)}),\quad
z\in\mathbb B_n,
\end{equation*}
whence we conclude that $\zeta_j^m=\zeta_k^m=\tilde\zeta$ for
$j,k=1,\dots,n$, i.e.~$\varphi$ is of the form (\ref{eq:lI}).

Finally, observe that (\ref{eq:U}) implies that
\begin{equation*}
f(\pi_n(z^l))=\pi_n(\tilde\zeta z^m),\quad z\in g_p(U).
\end{equation*}
Hence $f_n(z_1,\dots,z_n)=\tilde\zeta z_n^{m/l}$. Since $f$ is
holomorphic on $\mathbb E_{p,n}$, we conclude that $m/l\in\mathbb
N$.

\emph{Ad} (b). Assume now $m\in\mathbb N$. Then
$\pi_n\circ\varphi^m:\mathbb B_n\rightarrow\mathbb E_{1/m,n}$ is
the proper holomorphic mapping with multiplicity $n!m$. Thus
equality (\ref{eq:U}) extends on $\mathbb B_n$ and implies that
$g_{1/p}:\mathbb B_n\rightarrow\mathbb B_{p,n}$ is the proper
holomorphic mapping with multiplicity $1/p=l\in\mathbb N$, $m=kl$,
where $k\in\mathbb N$ is the multiplicity of $f$.

The equality (\ref{eq:U}) and the identity principle imply that
for any $\sigma\in\mathfrak S_n$ and
$\xi=(\xi_1,\dots,\xi_n)\in\mathbb T^n$, $\xi_j^l=1$,
$j=1,\dots,n$, there are $\tau\in\mathfrak S_n$ and
$\eta=(\eta_1,\dots,\eta_n)\in\mathbb T^n$, $\eta_j^m=1$,
$j=1,\dots,n$, such that
\begin{equation}\label{eq:sigmaxi}
\varphi(z)=\eta\varphi_{\tau}(\xi z_{\sigma}),\quad z\in\mathbb B_n.
\end{equation}

Observe that condition (\ref{eq:sigmaxi}) implies that
$a=(a_0,\dots,a_0)$ for some $a_0\in\frac{1}{\sqrt{n}}\mathbb
B_1$. Indeed, for any $\sigma\in\mathfrak S_n$ there are
$\tau\in\mathfrak S_n$ and $\eta\in\mathbb T^n$ such that
\begin{equation*}
0=\varphi(a)=\eta\varphi_{\tau}(a_{\sigma}).
\end{equation*}
Hence $\varphi(a_{\sigma})=0$, i.e. $a=a_{\sigma}$.

Moreover, for $l>1$ $\varphi$ is unitary. Indeed, suppose
$a\neq0$. Then there is $\xi\in\mathbb T^n$, $\xi^l=1$, with $\xi
a\neq a$. Hence $0=\varphi(a)=\eta\varphi_{\tau}(\xi
a)$---contradiction, since $\varphi_{\tau}(a)=0$.

\emph{Ad} (i). The equality (\ref{eq:U}) implies that
$z\mapsto\pi_n(\varphi^m(z))$ is symmetric polynomial mapping. In
particular, the polynomial
\begin{equation}\label{eq:pinn}
\mathbb
C^n\ni(z_1,\dots,z_n)\mapsto\prod_{j=1}^n\left(\sum_{k=1}^nq_{j,k}(z_k-a_0)\right)^m
\end{equation}
is symmetric.

Let $N_j:=\#\{k:q_{j,k}\neq0\}$, $j=1,2,\dots,n$, and let
\begin{equation*}
N_Q:=\min\{\#\{z:\exists_{j,k}\
\eta_jq_{j,k}=z\}:\eta_j,\in\mathbb T,\ \eta_j^m=1,\
j=1,\dots,n\}.
\end{equation*}
The matrix $Q{}^t\!\eta$ for $\eta=(\eta_1,\dots,\eta_n)\in\mathbb
T^n$, $\eta_j^m=1$, $j=1,\dots,n$, such that
$N_Q=\#\{z:\exists_{j,k}\ \eta_jq_{j,k}=z\}$ we call the
\emph{reduced} matrix of the matrix $Q$.
 Note that the polynomial (\ref{eq:pinn}) has $n$
different---up to multiplicative constant---linear factors.
Consequently, either
\begin{enumerate}
\item[(i-i)]$N_1=\dots=N_n=1$, or
\item[(i-ii)]$N_1=\dots=N_n=n-1$, or
\item[(i-iii)]$N_1=\dots=N_n=n$.
\end{enumerate}
We consider these cases separately.

\emph{Case} (i-i). If $N_j=1$, $j=1,2,\dots,n$, then for any $k$
there is a unique $j=j(k)$ such that $q_{j,k}\neq0$. Consequently,
\begin{equation*}
\varphi(z_1,\dots,z_n)=(q_{1,\sigma(1)}(z_{\sigma(1)}-a_0),\dots,q_{n,\sigma(n)}(z_{\sigma(n)}-a_0))
\end{equation*}
for some $\sigma\in\mathfrak S_n$. First equality in
(\ref{eq:stein}) implies that $a_0=0$ and $|q_{j,\sigma{(j)}}|=1$,
$j=1,\dots,n$. Repeating the argument from part (a) we conclude
that $\varphi$ is of the form (\ref{eq:lI}).

\emph{Case} (i-ii). Suppose now that $N_j=n-1\geqslant2$,
$j=1,2,\dots,n$. Then the symmetry of the polynomial
(\ref{eq:pinn}) implies that for any $k$ there is a unique $j$
such that $q_{j,k}=0$. We consider two cases.
\begin{itemize}
\item Assume that $N_Q=2$ and if $q_{j,k}\neq0$ then
$q_{j,k}=\eta_j\alpha$, $\eta_j\in\mathbb T$, $\eta_j^m=1$,
$j,k=1,\dots,n$, for some $\alpha\in\mathbb C_*$. First equality
in (\ref{eq:stein}) implies
\begin{equation}\label{eq:n-11}
(n-1)|\alpha|^2=1+|a_0|^2|R|^2,\quad (n-2)|\alpha|^2=|a_0|^2|R|^2.
\end{equation}
If $a_0=0$ then the equalities above lead to contradiction.
Therefore assume $a_0\neq0$. Consequently, $|\alpha|=1$, which,
together with the second equality in (\ref{eq:stein}) implies
\begin{equation}\label{eq:n-12}
|R|^2-(n-1)^3|a_0|^2=1.
\end{equation}
It follows from (\ref{eq:n-11}) and (\ref{eq:n-12}) that
\begin{equation*}
|a_0|^2=\frac{n(n-2)}{(n-1)^3}.
\end{equation*}
Elementary calculation shows that
$\frac{1}{n}<\frac{n(n-2)}{(n-1)^3}$ for $n\geqslant3$---a
contradiction, since $|a_0|^2<\frac{1}{n}$.
\item Assume now that $N_Q\geqslant3$. Then the symmetric polynomial
(\ref{eq:pinn}) has to have at least $2n$ different factors---a
contradiction.
\end{itemize}

\emph{Case} (i-iii). Let $N_j=n$, $j=1,2,\dots,n$. We consider
three cases.
\begin{itemize}
\item$N_Q=1$. Then the first equality in (\ref{eq:stein}) leads to a
contradiction.
\item$N_Q=2$ and one of the entries in the reduced matrix $Q{}^t\!\eta$ appears in some row exactly once.
Because of the symmetry of the polynomial (\ref{eq:pinn}) we infer
that it is the case in every row and in every column. Hence we may
assume that
\begin{equation}\label{eq:i-iii}
q_{j,k}=\begin{cases}\eta_j\alpha,\quad&\textnormal{if }j\neq
k\\\eta_j\beta,\quad&\textnormal{if }j=k
\end{cases},
\end{equation}
for some $\alpha,\beta\in\mathbb C$, $\alpha\neq\beta$. Then
equalities (\ref{eq:stein}) give
\begin{equation*}
\begin{cases}
(n-1)|\alpha|^2+|\beta|^2=1+|a_0|^2|R|^2\\
(n-2)|\alpha|^2+2\re(\alpha\bar\beta)=|a_0|^2|R|^2\\
|R|^2-n|a_0|^2|(n-1)\alpha+\beta|^2=1\\
\bar a_0|(n-1)\alpha+\beta|^2=a_0|R|^2
\end{cases},
\end{equation*}
which, after elementary calculation, implies
\begin{equation*}
\alpha=\frac{1}{n}\left(\frac{\zeta_1}{\sqrt{1-na_0^2}}+\zeta_2\right),\quad
\beta=\frac{1}{n}\left(\frac{\zeta_1}{\sqrt{1-na_0^2}}-(n-1)\zeta_2\right)
\end{equation*}
for some $a_0\in\mathbb R$, $a_0^2<\frac{1}{n}$, and
$\zeta_1,\zeta_2\in\mathbb T$. Consequently,
\begin{equation}\label{eq:Q1}
q_{j,k}=\begin{cases}\frac{\eta_j}{n}\left(\frac{\zeta_1}{\sqrt{1-na_0^2}}+\zeta_2\right),\quad&\textnormal{if
}j\neq
k\\\frac{\eta_j}{n}\left(\frac{\zeta_1}{\sqrt{1-na_0^2}}-(n-1)\zeta_2\right),\quad&\textnormal{if
}j=k
\end{cases}
\end{equation}
for some $a_0\in\mathbb R$, $a_0^2<\frac{1}{n}$, and
$\zeta_1,\zeta_2,\eta_j\in\mathbb T$, $\eta_j^m=1$, $j=1,\dots,n$.
Condition (\ref{eq:Q1}) implies that $\varphi$ is as in (i).
\item$N_Q\geqslant2$ and each of the entries in the reduced matrix $Q{}^t\!\eta$ appears in some row at
least twice. Because of the symmetry of the polynomial
(\ref{eq:pinn}) we infer that it is the case in every row. Denote
all entries of the reduced matrix $Q{}^t\!\eta$ by
$\alpha_1,\dots,\alpha_{N_Q}$ and let $N_{\alpha_j}$ denote the
number of entries of the matrix $Q{}^t\!\eta$ equal to $\alpha_j$
in any row. By the symmetry, the polynomial (\ref{eq:pinn}) has to
have $n!/\prod_{j=1}^{N_Q}N_{\alpha_j}!$ different factors. Since
$N_{\alpha_j}\geqslant2$, $j=1,\dots,N_Q$, we easily conclude that
\begin{equation*}
\frac{n!}{\prod_{j=1}^{N_Q}N_{\alpha_j}!}>n
\end{equation*}
--- a contradiction.
\end{itemize}

\emph{Ad} (ii). We repeat the reasoning from the case (i). Since
$a=0$, the case $N_j=1$, $j=1,\dots,n$, implies that $\varphi$ is
of the form (\ref{eq:lI}).

Suppose now that $N_j=N\in\{n-1,n\}$, $j=1,2,\dots,n$. Equality
(\ref{eq:U}) implies that $\pi_n(\varphi^m(z))=\pi_n(\varphi^m(\xi
z))$ for any $z\in\mathbb B_n$ and
$\xi=(\xi_1,\dots,\xi_n)\in\mathbb T^n$, $\xi_j^l=1$,
$j=1,\dots,n$, $l\geqslant2$.
\begin{itemize}
\item If $N=n$, then the polynomial (\ref{eq:pinn}) has
to have at least $2^{n-1}$ different---up to multiplicative
constant---linear factors. Consequently, $2^{n-1}\leqslant n$---a
contradiction;
\item If $N=n-1\neq1$, then the polynomial (\ref{eq:pinn}) has
to have at least $2n$ different---up to multiplicative
constant---linear factors. Consequently, $2n\leqslant n$---a
contradiction.
\end{itemize}

\emph{Ad} (iii). In this case condition (\ref{eq:sigmaxi}) means
that for any $\sigma\in\mathfrak S_2$ there are $\tau\in\mathfrak
S_2$ and $\eta=(\eta_1,\eta_2)\in\mathbb T^2$, $\eta_j^m=1$,
$j=1,2$, such that
\begin{equation}\label{eq:sigmaxi(iii)}
\varphi(z)=\eta\varphi_{\tau}(z_{\sigma}),\quad z\in\mathbb B_2.
\end{equation}
Without loss of generality we may assume that $\sigma\neq\id$. We consider two cases.

\emph{Case $\tau=\id$}. Then (\ref{eq:sigmaxi(iii)}) for
$z=(0,a_0)$ and $z=(a_0,0)$ implies $a_0=0$ or
\begin{equation}\label{eq:u1}
\begin{cases}q_{1,1}=\eta_1q_{1,2}\\q_{2,1}=\eta_2q_{2,2}\\q_{1,2}=\eta_1q_{1,1}\\q_{2,2}=\eta_2q_{2,1}\end{cases}.
\end{equation}
If $a_0=0$, then (\ref{eq:sigmaxi(iii)}) for $z=(0,z_2)$,
$z_2\neq0$, and $z=(z_1,0)$, $z_1\neq0$, implies again condition
(\ref{eq:u1}). Observe that $\eta_1^2=\eta_2^2=1$.

Consequently, we are looking for matrix $Q$ satisfying
(\ref{eq:stein}) and (\ref{eq:u1}). Elementary calculation shows
that
\begin{itemize}
\item if $m$ is odd, then $\eta_j=1$, $j=1,2$, and
(\ref{eq:u1}) leads to contradiction with (\ref{eq:stein});
\item if $m$ is even, then $Q$ satisfies (\ref{eq:stein}) iff $\eta_1\eta_2=-1$ and $a_0\in\mathbb R$, $a_0^2<\frac12$. In this case $Q$ is, up to permutation of the rows, of
the form
\begin{equation}\label{eq:Q2}
Q=\frac{1}{\sqrt{2}}\begin{pmatrix}\frac{\zeta_1}{\sqrt{1-2a_0^2}}&\frac{\zeta_1}{\sqrt{1-2a_0^2}}\\\zeta_2&-\zeta_2\end{pmatrix},\quad\zeta_1,\zeta_2\in\mathbb
T.
\end{equation}
\end{itemize}
Condition (\ref{eq:Q2}) implies that $\varphi$ is of the form
(\ref{eq:lIV}).

\emph{Case $\tau\neq\id$}. Then (\ref{eq:sigmaxi(iii)}) for
$z=(0,a_0)$ and $z=((a_0,0))$ implies $a_0=0$ or
\begin{equation}\label{eq:u2}
\begin{cases}q_{1,1}=\eta_1q_{2,2}\\q_{2,1}=\eta_2q_{1,2}\\q_{1,2}=\eta_1q_{2,1}\\q_{2,2}=\eta_2q_{1,1}\end{cases}.
\end{equation}
If $a_0=0$, then (\ref{eq:sigmaxi(iii)}) for $z=(0,z_2)$,
$z_2\neq0$, and $z=(z_1,0)$, $z_1\neq0$, implies again condition
(\ref{eq:u2}), which is equal to (\ref{eq:i-iii}). Consequently,
we infer that the matrix $Q$ is of the form (\ref{eq:Q1}), i.e.
$\varphi$ is of the form (\ref{eq:lII}).

\emph{Ad} (iv). We consider two cases.

\emph{Case $\sigma=\id$, $\xi=(1,-1)$}. It follows immediately from (\ref{eq:sigmaxi}) that $\tau\neq\id$. Then (\ref{eq:sigmaxi}) for
$z=(z_1,0)$, $z_1\neq0$, and $z=(0,z_2)$, $z_2\neq0$, implies
\begin{equation*}
\begin{cases}q_{1,1}=\eta_1q_{1,2}\\q_{2,1}=\eta_2q_{1,1}\\q_{1,2}=-\eta_1q_{2,2}\\q_{2,2}=-\eta_2q_{1,2}\end{cases}.
\end{equation*}
In particular, $\eta_1\eta_2=1$. Since $Q$ is unitary, it follows that
\begin{equation}\label{eq:Q6}
Q=\frac{1}{\sqrt{2}}\begin{pmatrix}\zeta_1&\zeta_2\\\eta\zeta_1&-\eta\zeta_2\end{pmatrix},\quad\zeta_1,\zeta_2,\eta\in\mathbb T,\ \eta^m=1.
\end{equation}

\emph{Case $\sigma\neq\id$, $\xi=(1,-1)$}. We consider two subcases.
\begin{itemize}
\item $\tau=\id$. Then (\ref{eq:sigmaxi}) for
$z=(z_1,0)$, $z_1\neq0$, and $z=(0,z_2)$, $z_2\neq0$, implies
\begin{equation*}
\begin{cases}q_{1,1}=-\eta_1q_{1,2}\\q_{2,1}=-\eta_2q_{2,2}\\q_{1,2}=\eta_1q_{1,1}\\q_{2,2}=\eta_2q_{2,1}\end{cases}.
\end{equation*}
In particular, $\eta_1^2=\eta_2^2=-1$, which is possible only iff
$4\mid m$. If it is the case, then $Q$ is unitary iff, up to
permutation of rows,
\begin{equation}\label{eq:Q7}
Q=\frac{1}{\sqrt{2}}\begin{pmatrix}\zeta_1&i\zeta_1\\\zeta_2&-i\zeta_2\end{pmatrix},\quad\zeta_1,\zeta_2\in\mathbb
T.
\end{equation}
\item $\tau\neq\id$. Then (\ref{eq:sigmaxi}) for
$z=(z_1,0)$, $z_1\neq0$, and $z=(0,z_2)$, $z_2\neq0$, implies
\begin{equation*}
\begin{cases}q_{1,1}=-\eta_1q_{2,2}\\q_{2,1}=-\eta_2q_{1,2}\\q_{1,2}=\eta_1q_{2,1}\\q_{2,2}=\eta_2q_{1,1}\end{cases}.
\end{equation*}
In particular, $\eta_1\eta_2=-1$, whence $\eta_1=-\bar\eta_2$.
then $Q$ is unitary iff, up to permutation of rows,
\begin{equation}\label{eq:Q8}
Q=\zeta\begin{pmatrix}r&\sqrt{1-r^2}\\-\eta\sqrt{1-r^2}&\eta r
\end{pmatrix},\quad\zeta,\eta\in\mathbb T,\ \eta^m=1,\ r\in[0,1].
\end{equation}
\end{itemize}

Straightforward calculation shows that the only unitary matrices
satisfying (\ref{eq:Q6}) and ((\ref{eq:Q7}) or (\ref{eq:Q8})) are,
up to permutation of rows, of the form
\begin{equation}\label{eq:Q9}
Q=\frac{\zeta}{\sqrt{2}}\begin{pmatrix}1&1\\\eta&-\eta\end{pmatrix},\quad\zeta,\eta\in\mathbb
T,\ \eta^m=1.
\end{equation}

It is easy to see that unitary automorphism $\varphi$ of $\mathbb
B_2$ represented by matrix (\ref{eq:Q9}) satisfies condition
(\ref{eq:sigmaxi}). Since (\ref{eq:sigmaxi}) is also satisfied by
the automorphism $\varphi$ of the form (\ref{eq:lI}), the proof of
the case (iv) is finished.
\end{proof}

\end{document}